\documentclass[12pt]{amsart}

\usepackage{tikz-cd}

\usepackage[english]{babel}
\usepackage[utf8x]{inputenc}
\usepackage[T1]{fontenc}
\usepackage{amsmath}
\usepackage{amstext}
\usepackage{amsfonts}
\usepackage{amssymb}
\usepackage{amsthm}

\usepackage[alphabetic, initials]{amsrefs}
\usepackage{mathtools}
\usepackage{dsfont}
\usepackage{color}
\usepackage{graphicx}
\usepackage[colorinlistoftodos]{todonotes}
\usepackage[colorlinks, linkcolor=red, citecolor=blue, urlcolor=blue, hypertexnames=true]{hyperref}

\usepackage{float}
\usepackage[colorinlistoftodos]{todonotes}
\usepackage{theoremref}
\usepackage{enumitem}
\usepackage{todonotes}

\allowdisplaybreaks
\newtheorem{theorem}{Theorem}[section]
\newtheorem{lemma}[theorem]{Lemma}
\newtheorem{prop}[theorem]{Proposition}
\newtheorem{cor}[theorem]{Corollary}
\newtheorem{thm}[theorem]{Theorem}
\newtheorem{lem}[theorem]{Lemma}

\newtheorem{question}[theorem]{Question}
\newtheorem*{cor*}{Corollary}
\newtheorem*{conjecture*}{Conjecture}
\newtheorem*{thm*}{Theorem}
\newtheorem*{lem*}{Lemma}

\newtheorem*{prop*}{Proposition}
\theoremstyle{definition}
\newtheorem{definition}[theorem]{Definition}

\newtheorem{example}[theorem]{Example}
\newtheorem*{defn*}{Definition}

\theoremstyle{remark}
\newtheorem{remark}[theorem]{Remark}

\newcommand{\M}{\mathcal{M}}

\newcommand{\cM}{\mathcal{M}}

\newcommand{\bC}{{\mathbb{C}}}
\newcommand{\E}{{\mathbb{E}}}

\newcommand{\bZ}{{\mathbb{Z}}}

\newcommand{\C}{\mathbb C}
\newcommand{\G}{\Gamma}

\title[]{On invariant von Neumann subalgebras rigidity property}

\author{Tattwamasi Amrutam}
\address{T. Amrutam, Department of mathematics, Ben Gurion University of the Negev,
P.O.B. 653, Be'er Sheva, 8410501, Israel}
\email{tattwamasiamrutam@gmail.com}

\author{Yongle Jiang*}
\address{Y. Jiang, School of Mathematical Sciences, Dalian University of Technology, Dalian, 116024, China}
\email{yonglejiang@dlut.edu.cn}

\thanks{$\ast$-corresponding author}

\date{\today}

\begin{document}

\begin{abstract}
We say that a countable discrete group $\Gamma$ satisfies the invariant von Neumann subalgebras rigidity (ISR) property if every $\Gamma$- invariant von Neumann subalgebra $\mathcal{M}$ in $L(\Gamma)$ is of the form $L(\Lambda)$ for some normal subgroup $\Lambda\lhd \Gamma$. We show many ``negatively curved" groups, including all torsion free non-amenable hyperbolic groups and torsion free groups with positive first $L^2$-Betti number under a mild assumption, and certain finite direct product of them have this property. We also discuss whether the torsion-free assumption can be relaxed. 
\end{abstract}

\subjclass[2020]{Primary 46L10; Secondary 20F67, 47C15}

\keywords{invariant von Neumann subalgebras, hyperbolic groups, first $L^2$-Betti number}

\maketitle

\tableofcontents
\section{Introduction and statement of main results}\label{sec: introduction}
Let $\Gamma$ be a discrete group. A subgroup $\Lambda\le\Gamma$ leads to an inclusion $L(\Lambda)\le L(\Gamma)$ at the group von Neumann algebra level and at the bigger level of crossed products, to an inclusion of $\mathcal{M}\rtimes\Lambda\le \mathcal{M}\rtimes\Gamma$ for $\Gamma$-von Neumann algebra $\mathcal{M}$. Many deep results have appeared in regards to determining the structure of an intermediate algebra associated with such inclusions (see e.g.,~\cites{choda 1978,izumi1998galois,cameron2016intermediate,cameron2019galois} etc). All these results can be thought of as certain ``rigidity phenomenon" associated with algebras which are invariant under the conjugation action of a subgroup of the unitary group $\mathcal{U}(\mathcal{H})$. Seen from this view point, we are interested in subalgebras $\mathcal{M}$ associated with the inclusion $\mathbb{C}\subset L(\Gamma)$ which are invariant under the conjugation action of the group $\Gamma$.

On the other hand, several studies have been devoted to understanding the properties of the group which reflect at the von Neumann algebras level. For example, Boutonnet-Carderi~ \cites{BC2017, BC2015} initiated the general study of maximal amenable von Neumann algebras associated with maximal amenable subgroups, which was motivated by Popa's seminal result \cite{Popa}.
These works motivated the second named author along with Skalski \cite{js2021maximal} to study maximal Haagerup property, where they obtained several examples where maximal Haagerup subalgebras come from maximal Haagerup subgroups. It was further continued in \cites{jiang2020maximal,jiang2021maximal}, where only maximality condition is reserved. From our point of view, these results can be interpreted as finding a class of ``rigid" von Neumann algebras in the sense that these von Neummann algebras inherit the property of the groups.   

In this paper, we focus on another such ``rigidity" property. Namely, the property of invariance which also translates from the group to the group von Neumann algebra. To be more precise, a normal subgroup (which is an invariant subgroup under the conjugation action)
$\Lambda\triangleleft\Gamma$ leads to a $\Gamma$-invariant subalgebra $L(\Lambda)\le L(\Gamma)$ under the conjugation. We say that a group has invariant subalgebra rigidity property if every $\Gamma$-invariant subalgebra comes from a $\Gamma$-invariant subgroup. The works of Alekseev-Brugger~\cite{AB} and Brugger's thesis~\cite{Brugger18} appear to be the first  initiating a meticulous study of $\Gamma$-invariant von Neumann subalgebras of $L(\Gamma)$. Using character rigidity techniques introduced in ~\cites{CrePet, Peterson14}, they proved that if $\Gamma$ is a lattice in a higher rank simple real Lie group $G$ with trivial center, then any non-trivial subfactor $\mathcal{M}$ of $L(\Gamma)$ has finite (Jones) index. 

The primary source of motivation for our work is the recent result of Kalantar-Panagoupolos~\cite{KalPan}. In this paper, the authors completely determined the $\Gamma$-invariant von Neumann subalgebras of $L(\Gamma)$ for irreducible lattices $\Gamma$ in the product of higher rank simple Lie groups. They showed that every $\Gamma$-invariant sub-algebra of $L(\Gamma)$ is of the form $L(\Lambda)$ for a normal subgroup $\Lambda\triangleleft\Gamma$.  In order to do so, they used the non-commutative Nevo-Zimmer theorem~\cite{BH19} which is a deep structural result associated with higher rank lattices (also see \cites{BBHP,bader2021charmenability}). These results can be viewed as non-commutative versions of Margulis' normal subgroup theorem.

Complementary to the above results, Chifan-Das~\cite{chifan2020rigidity} showed that whenever $\Gamma$ is a ``negatively curved" group (in the sense of \cite{CS}), e.g., a non amenable group that is either exact and acylindrically hyperbolic or has positive first $L^2$-Betti number, then all $\Gamma$-invariant subfactors are commensurable to subalgebras $L(\Lambda)$ arising from normal subgroup $\Lambda\lhd \Gamma$. Their proof relies heavily on the deformation/rigidity techniques for array/quasi-cocycles on groups which were introduced and studied in \cites{CKP, CS, CSU1, CSU2}.
Motivated by this, we make the following definition.
\begin{definition}
\thlabel{ncns}
Let $\Gamma$ be a countable discrete group. We say that $\Gamma$ satisfies the invariant von Neumann subalgebras rigidity (abbreviated as ISR) property if every $\Gamma$-invariant von Neumann subalgebra in $L(\Gamma)$ is of the form $L(\Lambda)$ for some normal subgroup $\Lambda\lhd \Gamma$. 
\end{definition}
In this paper, we show that many ``negatively curved groups" satisfy the ISR property. In particular, we show the following.
\begin{theorem}
\thlabel{invariantsubalgebras}
Let $\Gamma$ be a torsion free non-amenable hyperbolic group. Then $\Gamma$ satisfies the ISR property, i.e., 
if $\mathcal{M}$ is a $\Gamma$-invariant von Neumann subalgebra of $L(\Gamma)$, then, $\mathcal{M}=L(\Lambda)$ for some normal subgroup $\Lambda\triangleleft\Gamma$.
\end{theorem}

Moreover, we also have the following theorem.
\begin{theorem}
\thlabel{invariantsubalgebras_2}
Let $\Gamma$ be a torsion free group with positive first $L^2$-Betti number. Assume that $\Gamma$ satisfies the condition $(*)$ defined in \cite{PetersonThom2011}, i.e., every non-trivial element of $\bZ \Gamma$ acts without kernel on $\ell^2\Gamma$. Then $\Gamma$ satisfies ISR property.
\end{theorem}

In fact, a finite direct products of such groups also satisfy the ISR property.

\begin{theorem}
\thlabel{invariantsubalgebras_3}
Let $n\geq 2$ be an integer and $\Gamma_1,\ldots, \Gamma_n$ be groups. If all $\Gamma_i$'s are groups as in \thref{invariantsubalgebras} or \thref{invariantsubalgebras_2}, then $\Gamma:=\Gamma_1\times \cdots\times \Gamma_n$ satisfies the ISR property.
\end{theorem}
An immediate consequence of \thref{invariantsubalgebras_2} and \thref{invariantsubalgebras_3} is that for groups $\Gamma$ considered here, $L(\Gamma)$ does not admit any non-trivial $\Gamma$ invariant Cartan subalgebras (see \thref{noinvcartan}). However, the picture is far from being complete. There has been a substantial progress in showing the absence of Cartan subalgebras in certain group von Neumann algebras, see e.g., \cites{popa83, voiculescu, ozawapopa, CS, CSU1, DI} and the reference therein. In particular, Chifan-Sinclair showed that for every i.c.c. hyperbolic group $\Gamma$, $L(\Gamma)$ has no Cartan subalgebras ~\cite{CS}. Nevertheless, it is a well-known open problem to determine whether $L(\Gamma)$ admits any Cartan subalgebras or not for a group $\Gamma$ with positive first $L^2$-Betti number.

We  now compare our results with the known ones and discuss our proof techniques and strategy.

To begin with, it is worthwhile to note that groups in \thref{invariantsubalgebras} and \thref{invariantsubalgebras_2} lie on the other end of the spectrum as compared to those of the lattices dealt in the works of \cites{AB, Brugger18, KalPan}. The groups we consider are similar to those considered by Chifan-Das~\cite{chifan2020rigidity}. In fact, under the additional assumption of i.c.c., the proof in~\cite[Corollary 3.17]{chifan2020rigidity} shows that for the groups $\Gamma$ considered there, every $\Gamma$-invariant subfactor $\cM$ is of the form $L(\Lambda)$ for some normal subgroup $\Lambda\lhd \Gamma$ (see Remark~\ref{remark: Chifan-Das work} and \thref{obsfromCD}). 
Nevertheless, there are still notable differences between ours and \cite{chifan2020rigidity} both in the statement of theorems and the method for proofs.

\subsection{Proof Strategy and techniques}
We now discuss our strategy for the proof, which is similar to that of the second named author's employed in \cites{jiang2020maximal, jiang2021maximal}.

Since $L(\Gamma)$ is a finite von Neumann algebra, every $\Gamma$-invariant von Neumann subalgebra $\mathcal{M}$ lies in the image of the unique trace preserving normal conditional expectation $\mathbb{E}_{\mathcal{M}}$. In order to prove our theorem, we need to show that $\lambda(g^{-1})\mathbb{E}_{\mathcal{M}}(\lambda(g))\in \mathbb{C}$ for all $g\in \Gamma$. In order to do so, we think of $\{\mathbb{E}_{\mathcal{M}}(\lambda(g)): g\in \Gamma\}$ as a set of unknowns and find enough equations in terms of their Fourier expansion to completely solve them. For our strategy to work, we need to completely determine $\mathbb{E}_{\cM}(\lambda(g))$, which clearly lies in $L(\langle g\rangle)'\cap L(\Gamma)$. We do so in two steps. First, we choose $s\in \Gamma$ appropriately in order to make $g$ and $sgs^{-1}$ free from each other. And then, we compute the Fourier expansion of the product $\mathbb{E}_{\cM}(\lambda(g))\cdot E_{\M}\left(\lambda(sgs^{-1})\right)$  and compare it with $\mathbb{E}_{\cM}(\lambda(g)\cdot \mathbb{E}_{\cM}\left(\lambda(sgs^{-1}))\right)$ using the bimodular property of $\mathbb{E}_{\cM}$.
\subsection{Organization of the paper}
The rest of the paper is organized as follows: In Section~\ref{sec: preliminaries}, we introduce the objects of our interest and fix some notation. We prove \thref{abstractcond} which gives an abstract condition for groups that entails the ISR property.  We also collect all the necessary facts on hyperbolic groups and groups with positive L$^2$-Betti number which are needed for our purposes. We also prove \thref{same_power} on free groups, which is used for comparing Fourier coefficients.
After which, we observe a simple necessary condition for having ISR property in Section~\ref{sec: a necessary condition}, which is used to construct groups without ISR property. The main theorems are proved in Section~\ref{sec: proof of theorems}. Finally, in Section~\ref{sec: remaining questions},  we give an example of a group for which \thref{abstractcond} is not applicable and yet possesses the ISR property. 

\section{Preliminary technicalities}\label{sec: preliminaries}
Throughout this paper, $\Gamma$ is going to be a discrete group. We denote by $e$ the neutral element in the group $\Gamma$. We briefly recall the construction of the group von Neumann algebra $L(\Gamma)$ and refer the reader to \cites{PopaDelaroche} for more details.

\subsection*{Group von Neumann algebra}
Let $\Gamma$ be a discrete group. Denote by $\ell^2(\Gamma)$ the Hilbert space of square summable functions on $\Gamma$, i.e.,
\[\ell^2(\Gamma)=\left\{f:\Gamma\to\mathbb{C}: \sum_{s\in\Gamma}|f(s)|^2<\infty\right\}.\]
We can now define an unitary representation of the group $\Gamma$ on $\ell^2(\Gamma)$ by left translations. This representation is usually called the left regular representation and is denoted by $\lambda$. More precisely,
$\lambda:\Gamma\to \mathbb{B}(\ell^2(\Gamma))$ is defined by $$\lambda(s)(\delta_t)=\delta_{st},~s,t\in\Gamma.$$
The group von Neumann algebra $L(\Gamma)$ is defined as the closure (inside $\mathbb{B}(\ell^2(\Gamma))$) of the set spanned by $\lambda(s)$'s under weak operator topology, namely.,
\[L(\Gamma)=\overline{\text{Span}\left\{\lambda(s): s\in \Gamma\right\}}^{w.o.t.}.\]
$L(\Gamma)$ also comes equipped with a normal faithful trace $\tau$ defined by 
\[ \tau(\lambda(s)) = \begin{cases} \begin{array}{ll}
        1 & \mbox{if $s=e$};\\
        0 & \mbox{otherwise}.\end{array} \end{cases} \]
For any element $x\in L(\G)$, we write \[\text{supp}(x)=\{g\in \G:~\tau(x\lambda(g)^*)\neq 0\}.\]
Note that we have a natural embedding $L(\Gamma)\hookrightarrow \ell^2(\Gamma)$ via the map $x\mapsto x\delta_e$. Hence, every $x\in L(\Gamma)$ can be written as $x=\sum_{g\in \Gamma}x_g\lambda(g)$, where  $\lambda(g)\in L(\Gamma)$ are the canonical unitaries of $L(\Gamma)$. The scalars $x_g$ are called the Fourier coefficients of $x$. Note that in the above sum, the convergence is in $\ell^2$-norm ($||\cdot ||_2$) and not with respect to the strong operator or weak operator topology, see e.g., \cite[Remark 1.3.7]{PopaDelaroche}. The above expansion is usually called the Fourier expansion of $x$.

\subsection*{Conditional expectation}
Let $\mathcal{M}\subseteq (L(\Gamma), \tau)$ be a von Neumann subalgebra. Then, there is a unique trace preserving conditional expectation (see \cite[Theorem~9.1.2]{PopaDelaroche}) $\mathbb{E}_{\mathcal{M}}: L(\Gamma)\to \mathcal{M}$, i.e., a linear map which satisfies the following properties:
\begin{enumerate}
    \item $\mathbb{E}_{\mathcal{M}}$ is positive, i.e., it maps positive elements in $L(\Gamma)$ to positive elements in $\mathcal{M}$.
    \item $\mathbb{E}_{\mathcal{M}}(x)=x$ for all $x\in \mathcal{M}$.
    \item $\mathbb{E}_{\mathcal{M}}$ satisfies the $\mathcal{M}$-bimodular property, i.e., 
    $$\mathbb{E}_{\mathcal{M}}(x_1yx_2)=x_1\mathbb{E}_{\mathcal{M}}(y)x_2,~\forall x_1, x_2\in \mathcal{M} \text{ and } y\in L(\Gamma).$$
\end{enumerate}
\noindent
In our approach, we  often need to know the Fourier expansion of $\mathbb{E}_{\mathcal{M}}(\lambda(g))\cdot \mathbb{E}_{\mathcal{M}}(\lambda(h))$. Moreover, if we assume that $\text{supp}(\mathbb{E}_{\mathcal{M}}(\lambda(g)))$ and $\text{supp}(\mathbb{E}_{\mathcal{M}}(\lambda(h)))$ satisfy the unique product property, i.e., $st=s't'$ iff $(s, t)=(s', t')$ for any $s, s'\in \text{supp}(\mathbb{E}_{\mathcal{M}}(\lambda(g)))$ and $t, t'\in \text{supp}(\mathbb{E}_{\mathcal{M}}(\lambda(h)))$, then 
\[\text{supp}(\mathbb{E}_{\mathcal{M}}(\lambda(g))\cdot \mathbb{E}_{\mathcal{M}}(\lambda(h)))=\text{supp}(\mathbb{E}_{\mathcal{M}}(\lambda(g)))\cdot \text{supp}(\mathbb{E}_{\mathcal{M}}(\lambda(h))).\] In other words, if we set $I=\text{supp}(\mathbb{E}_{\mathcal{M}}(\lambda(g)))$, $J=\text{supp}(\mathbb{E}_{\mathcal{M}}(\lambda(h)))$ and
write the Fourier expansion
\begin{align*}
    \mathbb{E}_{\mathcal{M}}(\lambda(g))=\sum_{s\in I}x_s\lambda(s)~\text{and}~
    \mathbb{E}_{\mathcal{M}}(\lambda(h))=\sum_{t\in J}y_t\lambda(t),
\end{align*}
then the Fourier expansion of 
$\mathbb{E}_{\mathcal{M}}(\lambda(g))\cdot \mathbb{E}_{\mathcal{M}}(\lambda(h))$ is exactly given by
\begin{align*}
    \mathbb{E}_{\mathcal{M}}(\lambda(g))\cdot \mathbb{E}_{\mathcal{M}}(\lambda(h))
    =\sum_{s\in I, t\in J}x_sy_t\lambda(st).
\end{align*}
Note that the above sum is convergent in the $\ell^2$-norm.
To guarantee that the unique product property holds, we usually take $h$ to be free from $g$ and at the same time assume nice control on the supports $I$ and $J$. Moreover, it is not hard to check that for any $a, b\in L(\Gamma)$, we have
\[\text{supp}(\mathbb{E}_{\mathcal{M}}(\mathbb{E}_{\mathcal{M}}(a)b)))\subseteq \bigcup_{g\in \text{supp}(\mathbb{E}_{\mathcal{M}}(a))}\text{supp}(\mathbb{E}_{\mathcal{M}}(\lambda(g)b)).\]

We begin with the following simple lemma which allows us to compare Fourier coefficients of elements in $L(\Gamma)$. 
\begin{lem}
\thlabel{same_power}
Suppose that $a$ and $b$ are free elements in $\mathbb{F}_2$, i.e., $\langle  a, b\rangle\cong \langle a\rangle *\langle b\rangle\cong \mathbb{F}_2$. If $i, j, \ell, n$ are nonzero integers such that $(a^kb^{\ell})^i=(a^nb^{\ell'})^j$ for some integer $\ell'$ and $k$, then $k=n$, $\ell=\ell'$ and $i=j$.
\end{lem}
\begin{proof}
Below, we work inside $\langle a, b\rangle$.
Denote by $\pi:~\langle  a, b\rangle\twoheadrightarrow\frac{\langle a, b\rangle}{[\langle a, b\rangle, \langle a, b\rangle]}\cong\mathbb{Z}^2$ the quotient map onto its abelianization. We now observe that $\pi(a)^{ki}\pi(b)^{\ell i}=\pi(a)^{nj}\pi(b)^{\ell' j}$. Thus, $ki=nj$ and $\ell i=\ell' j$. Since $k\ne 0$, it suffices to show that $k=n$ for this will imply that $i=j$ and $\ell=\ell'$. 

If $i>0$, then the initial $k+1$ letters (w.r.t. the generating set $\{a^{\pm}, b^{\pm}\}$ for $\langle a, b\rangle$) of the reduced word $(a^kb^{\ell})^i$ are just $a^kb^{\pm}$. Hence, $j>0$ and clearly this shows that $n=k$.

On the other hand, if $i<0$, then the final $k+1$ letters of the reduced word $(a^kb^{\ell})^i$ is $b^{\pm}a^{-k}$. Therefore, $j<0$ and hence, it follows that $n=k$.
\end{proof}

We give two abstract conditions from which \thref{invariantsubalgebras} follows as a corollary. 
We say that a nontrivial element $h\in \Gamma$ is primitive if the centralizer $C(h)=\langle h\rangle$.
\begin{prop}
\thlabel{abstractcond}
Let $\Gamma$ be a torsion-free discrete group satisfying the following two conditions:
\begin{enumerate}
    \item \label{condition-1}for any nontrivial $g \in \Gamma$, we can find some primitive $h\in \Gamma$ and nonzero integer $n$ such that $g=h^n$. Moreover there is some $s$ in $\G$ such that $h$ and $shs^{-1}$ are free, i.e., they generate a copy $\mathbb{F}_2$ in $\Gamma$.

\item\label{condition-2} any nontrivial primitive element $h$ in $\G$ generates a maximal abelian von Neumann subalgebra (masa) in $L(\G)$, i.e., 
\[L(\langle h\rangle) '\cap L(\G)=L(\langle h\rangle). \]
\end{enumerate}
Then, $\Gamma$ satisfies the ISR property.
\end{prop}

Let $\mathcal{M}$ be a $\Gamma$-invariant von Neumann subalgebra of $L(\G)$. Let $\mathbb{E}_{\mathcal{M}}$ be the $\Gamma$-equivariant conditional expectation onto $\mathcal{M}$. We let \[\Lambda=\{g\in \G:\mathbb{E}_{\mathcal{M}}(\lambda(g))\ne 0\}.\]
We show that $\mathbb{E}_{\mathcal{M}}(\lambda(g))=a_g\lambda(g)$ for some $a_g\in\mathbb{C}$ for every $g\in \G$. Note that this entails $\Lambda$ to be a subgroup of $\Gamma$. Moreover, we obtain that either $a_g=0$ or $1$ by applying $E_{\mathcal{M}}$ on both sides again. 
We now show that this implies that $\mathcal{M}=L(\Lambda)$. Clearly, $L(\Lambda)\subset\M$. For the other inclusion, observe that if $x\in \mathcal{M}$ and $s\in \text{supp}(x)$, then $0\neq \tau(x\lambda(s)^*)=\tau(\E_{\mathcal{M}}(x)\lambda(s)^*)=\tau(x\E_{\mathcal{M}}(\lambda(s)^*))=\tau(x\E_{\mathcal{M}}(\lambda(s))^*)$. Hence, $\E_{\mathcal{M}}(\lambda(s))\neq 0$, i.e., $s\in \Lambda$. Thus, $\mathcal{M}\subseteq L(\Lambda)$.

We would like to remark that such problems have been tackled in the past by showing that the relative commutant $\mathcal{M}'\cap L(\Gamma)=\mathbb{C}$, from where it follows easily that $\mathbb{E}_{\mathcal{M}}(\lambda(g))\lambda(g)^*\in\mathbb{C}$ (see e.g.,~\cite{chifan2020rigidity}). However, as pointed out in the introduction, our approach is to view $\{\mathbb{E}_{\mathcal{M}}(\lambda(g)):~g\in \Gamma\}$ as unknowns and find enough equations to completely determine them. 
\begin{proof}[Proof of \thref{abstractcond}]

We pick an element $g\ne e$ from $\Lambda$. Note that if no such $g$ exists, then $\mathcal{M}=\mathbb{C}$. Using Condition~(\ref{condition-1}), we may write $g=h^n$ for some primitive $h\in\G$ and $0\neq n\in\mathbb{Z}$. Moreover, we may find some $s\in \Gamma$ such that $h$ is free from $shs^{-1}$, i.e., $\langle h, shs^{-1}\rangle\cong\langle h\rangle *\langle shs^{-1}\rangle\cong \mathbb{F}_2$. 

For ease of notations, we write $a=h$ and $b=shs^{-1}$. Then, $a^n=g=h^n$ and $b^n=sh^ns^{-1}=sgs^{-1}$. 

Since $$\lambda(a^k)\mathbb{E}_{\mathcal{M}}(\lambda(a^n))\lambda(a^{-k})=\mathbb{E}_{\mathcal{M}}(\lambda(a^n)),~\forall k\in \mathbb{Z},$$ it follows that $\mathbb{E}_{\mathcal{M}}(\lambda(a^n))\in L\left(\langle a\rangle\right)'\cap L(\Gamma)$. Combining this observation with Condition~(\ref{condition-2}), we see that $\mathbb{E}_{\mathcal{M}}(\lambda(a^n))\in L\left(\langle a\rangle\right)$. As such, we can write 
\[\mathbb{E}_{\mathcal{M}}(\lambda(a^n))=\sum_{k\in\mathbb{Z}}c_k\lambda(a^k),~c_k\in\mathbb{C}.\]
Similarly, we can write 
\[\mathbb{E}_{\mathcal{M}}(\lambda(b^n))=\sum_{\ell\in\mathbb{Z}}d_{\ell}\lambda(b^{\ell}),~d_{\ell}\in\mathbb{C}.\]
Using the fact that $$\mathbb{E}_{\mathcal{M}}\left(\lambda(a^n)\mathbb{E}_{\mathcal{M}}(\lambda(b^n))\right)=\mathbb{E}_{\mathcal{M}}\left(\lambda(a^n)\right)\mathbb{E}_{\mathcal{M}}\left(\lambda(b^n))\right),$$
we obtain that
\begin{align*}
\sum_{k,\ell\in\mathbb{Z}}c_kd_{\ell}\lambda(a^kb^{\ell})
=\mathbb{E}_{\mathcal{M}}(\lambda(a^n)\sum_{\ell'\in\mathbb{Z}}d_{\ell'}\lambda(b^{\ell'}))
=\sum_{\ell'\in\mathbb{Z}}d_{\ell'}\mathbb{E}_{\mathcal{M}}(\lambda(a^nb^{\ell'})).
\end{align*}
So this implies that if $c_kd_{\ell}\neq 0$, then
\begin{align*}
a^kb^{\ell}&\in  \cup_{\ell'\in\bZ}\text{supp}(\mathbb{E}_{\mathcal{M}}(\lambda(a^nb^{\ell'})))\\
&\subseteq \cup_{j\in \bZ}\{h^j:~h~\mbox{is primitive and}~\exists \ell'\in \mathbb{Z}~\mbox{and}~i\neq 0,~\mbox{s.t.}~a^nb^{\ell'}=h^i\}.
\end{align*}
Note that in order to get the last inclusion, we have used the fact that if $h$ is primitive with $a^nb^{\ell'}=h^i$ for some $i\neq 0$, then $\text{supp}(\mathbb{E}_{\mathcal{M}}(\lambda(a^nb^{\ell'})))\subseteq \{h^j:~j\in\mathbb{Z}\}$ by condition~$(2)$.

\textit{Claim:} for any $k\neq n$ and $\ell\neq 0$, we have $c_kd_{\ell}=0$.

Suppose otherwise. Then, $a^kb^{\ell}=h^j$ for some primitive $h$ with $a^nb^{\ell'}=h^i$. Clearly, $j\neq 0$. Therefore, $(a^kb^{\ell})^i=(a^nb^{\ell'})^j$. Applying Lemma \ref{same_power}, we deduce that $k=n$. This is a contradiction.

If there exists $k\in\mathbb{Z}\setminus \{n\}$ such that $c_k\ne 0$, then $d_{\ell}=0$ for all $\ell\ne 0$. This will force $\mathbb{E}_{\mathcal{M}}(\lambda(b^n))=d_0$. We can then apply the canonical trace $\tau$ on both sides and use the fact that $\mathbb{E}_{\mathcal{M}}$ is $\tau$-invariant to conclude that $d_0=0$. This will contradict the assumption of $\mathbb{E}_{\mathcal{M}}(\lambda(b^n))$ being non-zero. As a result, we obtain that $c_k=0$ for all $k\ne n\in \mathbb{Z}$. This precisely tells us that 
$\mathbb{E}_{\mathcal{M}}(\lambda(a^n))=c_n\lambda(a^n)$ and we are done. 
\end{proof}
We now briefly recall all the necessary properties of groups which we shall put to use later.

We begin with the definition of property-naive. 
\begin{definition}[property-naive]
\thlabel{propertynaive}
A discrete group $\G$ is said to have property-naive (denoted by $P_{\text{nai}}$) if for any finite subset
$F$ of $\G\setminus\{1\}$ there exists an element $g_0\in\G$ of infinite order such that for each
$s \in F$, the subgroup $\langle s, g_0\rangle$ of $\G$, generated by $s$ and $g_0$, is canonically isomorphic
to the free product $\langle s\rangle *\langle g_0\rangle$.
\end{definition}
Property $P_{\text{nai}}$ has been exploited in the past to show that hyperbolic groups with trivial amenable radical are $C^*$-simple (see e.g.,~\cites{bekka1994some, arzhantseva2006relatively, AbbDah2019} etc). In particular, torsion free non-amenable hyperbolic groups, or more generally any non-amenable acylindrically hyperbolic groups with no non-trivial finite normal subgroups, possess $P_{\text{nai}}$.

We also recall the notion of ``unique root" property for later use.
\begin{definition}[Unique root property]
\thlabel{uniqueroot}
A group $\Gamma$ is said to have unique root property if for any $s,t\in\Gamma$ and for any positive integer $n$, the equality $s^n=t^n$ implies that $s=t$.
\end{definition}
It is well known that in torsion free word-hyperbolic groups, nontrivial elements have cyclic centralizers, see e.g. \cite[Page 462-463]{BriHae1999}. Moreover, if $\Gamma$ is a torsion free group with the property that every nontrivial element has cyclic centralizers, then it satisfies the ``unique root" property (see e.g., \cite[Lemma~2.2]{bartholdi2010abstract}).
We now show that these groups satisfy the conditions of \thref{abstractcond}.
\begin{prop}
\thlabel{conditionsforabstractcond}
Let $\Gamma$ be a torsion-free group with the property that every nontrivial element has cyclic centralizer. Also, assume that $\Gamma$ has property $P_{\text{nai}}$. Then, $\Gamma$ satisfies the following conditions:
\begin{enumerate}
    \item for any nontrivial $g \in \Gamma$, there is a primitive element $h\in \Gamma$ and nonzero integer $n$ such that  $g=h^n$. Moreover, there is some $s$ in $\G$ such that $h$ and $shs^{-1}$ are free, i.e., they generate a copy $\mathbb{F}_2$ in $\Gamma$.

\item for any nontrivial primitive element $h$ in $\G$, $L(\langle h\rangle)$ is a masa in $L(\G)$, i.e., 
\[L(\langle h\rangle) '\cap L(\G)=L(\langle h\rangle). \]
\end{enumerate}
\begin{proof}
The existence of $s$ in Condition~\ref{condition-1} is a consequence of property $P_{\text{nai}}$. We now proceed to verify the existence of primitive $h$ with $g=h^n$. Since $C(g)$ is cyclic, we may write $C(g)=\langle h\rangle$ for some $h\in \Gamma$. Since $g\in C(g)$, we may assume $g=h^n$. Now, we argue that $C(h)=\langle h\rangle$. Indeed, $\supseteq$ is clear. For the converse, note that $C(h)\subseteq C(g)=\langle h\rangle$.

We now prove that Condition~\ref{condition-2} is a consequence of the unique root property. Towards this end, let $h\in \Gamma$ be a primitive element. It is enough to show that for any  $s\in \Gamma\setminus \langle h\rangle$, $\sharp\{h^{-i}sh^i:~i\in\mathbb{Z}\}=\infty$.
Note that this implicitly implies that $h$ has infinite order. 
Let $C(s)$ denote the centralizer of $s$ in $\Gamma$. Equivalently, we need to show that for any $s\not\in \langle h\rangle$, we have $\langle h\rangle \cap C(s)=\{e\}$. Towards a contradiction, suppose otherwise. Then, $h^i\in \langle h\rangle \cap C(s)$ for some $i\neq 0$. This implies that $sh^i=h^is$. Replacing $i$ by $-i$ if required, we may assume that $i>0$. Now, this shows that $(s^{-1}hs)^i=h^i$ for some $i\in\mathbb{N}$. Since $\Gamma$ has unique root property, it follows that $s^{-1}hs=h$. Hence, $s\in C(h)=\langle h\rangle$ which is a contradiction. Therefore, the claim follows.
\end{proof}
\end{prop}
\begin{remark}\thlabel{remark_Dixmier's work}
The way we prove Condition 2 has its root in Dixmier's work \cite{Dixmier}, where he found some sufficient conditions for a maximal abelian subgroup to generates a masa in the ambient group von Neumann algebra.
\end{remark}
Recall that given a subgroup inclusion $A\subset \Gamma$, the virtual centralizer of $A$ in $\Gamma$ denoted by $vC_{\Gamma}(A)$ is the subgroup of all elements $g\in \Gamma$ whose $A$-orbit under conjugation is finite. We observe that the proof of \thref{conditionsforabstractcond} shows the following.

\begin{lem}
\thlabel{control_support_in_relative_commutant}
Let $\Gamma$ be a countable discrete group and let $A\subseteq \Gamma$ be a subgroup. Then, the relative commutant  satisfies $L(A)'\cap L(\Gamma)\subseteq L(vC_{\Gamma}(A))$. In particular, for any infinite order element $g\in \Gamma$, if $x\in L(\langle  g\rangle)'\cap L(\Gamma)$, then $supp(x)\subseteq \cup_{i\geq 1}C(g^i)$.
\end{lem}
\begin{proof}
Let $x\in L(A)'\cap L(\Gamma)$ and write $x=\sum_{g\in \Gamma}x_g \lambda(g)$ for its Fourier 
expansion. Then we know that $x_g=x_{aga^{-1}}$ for all $a\in A$. 
Since $||x||^2_2=\sum_{g\in \Gamma}|x_g|^2<\infty$, we deduce that $x_g=0$ for all $g\not\in vC_{\Gamma}(A)$, i.e., $\text{supp}(x)\subseteq vC_{\Gamma}(A)$, hence $L(A)'\cap L(\Gamma)\subseteq L(vC_{\Gamma}(A))$. 

The last part follows by taking $A=\langle g \rangle$ and observing that $vC_{\Gamma}(A)=\cup_{i\geq 1}C(g^i)$.
\end{proof}
\begin{lemma}
\thlabel{malnormal}
Let $\Gamma$ be a discrete group. Suppose that $B\subset \Gamma$ is malnormal, i.e., $gBg^{-1}\cap B=\{e\}$ for all $g\in \Gamma\setminus B$. Then, for any infinite subgroup $A\subseteq B$, the relative commutant $L(A)'\cap L(\Gamma)\subseteq L(B)$.
\end{lemma}
\begin{proof}
Note that malnormality condition implies $vC_{\Gamma}(A)\subseteq B$, then we can apply \thref{control_support_in_relative_commutant} to finish the proof.
\end{proof}

For a countable discrete group $\Gamma$, its 
 first $L^2$-Betti number,  written as $\beta_1^{(2)}(\Gamma)$ is defined to be a certain dimension of either $H_1(\Gamma,\ell^2(\Gamma))$ or $H^1(\Gamma,\ell^2(\Gamma))$, see e.g., \cites{lueck, PetersonThom2011}.
We will not be needing the above definition as it is. For our purposes, it suffices to use a certain property called ``Property~$(*)$" and a striking theorem on groups having positive first $L^2$-Betting number~\cite{PetersonThom2011}. We first recall the definition of property~$(*)$.
\begin{definition}[Property~$(*)$]
\thlabel{prop*}
Let $\Gamma$ be a countable discrete group. We say that $\Gamma$ has property~$(*)$ if every non-trivial element of $\bZ \Gamma$ acts without kernel on $\ell^2\Gamma$.
\end{definition}
\begin{thm}[Theorem 4.1 in \cite{PetersonThom2011}]
Let $\Gamma$ be a torsion free countable discrete group. There exists a
family of subgroups $\{\Gamma_i:~i\in I\}$, such that
\begin{enumerate}
\item[(i)] We can write $\Gamma$ as the disjoint union:
\[\Gamma=\{e\}\cup \bigcup_{i\in I}\dot{\Gamma_i},\]
where, $\dot{\G_i}=\G_i\setminus \{e\}$.
\item[(ii)] The groups $\G_i$ are mal-normal in $\Gamma$, for $i\in I$.
\item[(iii)] If $\G$ satisfies condition~$(*)$, then $\Gamma_i$ is free from $\G_j$, for $i\neq j$.
\item[(iv)] $\beta_1^{(2)}(\G_i)=0$, for all $i\in I$.
\end{enumerate}
\end{thm}
As mentioned in \cite[Page 574-575]{PetersonThom2011}, it is known that all right orderable groups and all residually torsion free
elementary amenable groups satisfy this condition (with no known counterexamples).

\section{A necessary condition for groups with ISR property}\label{sec: a necessary condition}

In this section, we make a simple observation on necessary conditions for groups with ISR property.

\begin{prop}
\thlabel{necessary_condition_for_ISR}
Let $\Gamma$ be a countable discrete group with the ISR property. Then the finite conjugacy radical of $\Gamma$ has at most two elements. If $\Gamma$ is further assumed to be infinite, then $\Gamma$ is i.c.c.
\end{prop}
\begin{proof}
Let $\Gamma_{fin}$ be the finite conjugacy radical of $\Gamma$, i.e., it is the normal subgroup of $\Gamma$ consisting of all elements with finite conjugacy classes.

First, we show that  $C(\Gamma)=\Gamma_{fin}$ and $L(C(\Gamma))=\mathcal{Z}(L(\Gamma))=L(\Gamma_{fin})$, where $C(\Gamma)$ denotes the center of $\Gamma$.

Since $\mathcal{Z}(L(\Gamma))$ is $\Gamma$-invariant, we deduce that $\mathcal{Z}(L(\Gamma))=L(\Lambda)$ for some normal subgroup $\Lambda\lhd \Gamma$ from the ISR property.
Clearly, we have $L(C(\Gamma))\subseteq \mathcal{Z}(L(\Gamma))\subseteq L(\Gamma_{fin})$, thus, $C(\Gamma)\subseteq \Lambda\subseteq \Gamma_{fin}$. Moreover, since $\Lambda\subseteq L(\Lambda)=\mathcal{Z}(L(\Gamma))$, we deduce that $\Lambda\subseteq C(\Gamma)$. Thus, $C(\Gamma)=\Lambda$ and $\mathcal{Z}(L(\Gamma))=L(C(\Gamma))$. Next, for any $g\in \Gamma_{fin}$, we list the elements in its conjugacy class as $s_i^{-1}gs_i$, where $1\leq i\leq n$ for some $n\in \mathbb{N}$ with $s_1=e$. Then $x:=\sum_{i=1}^n\lambda(s_i^{-1}gs_i)\in\mathcal{Z}(L(\Gamma))=L(\Lambda)$, hence $g\in \Lambda$. Therefore, $\Gamma_{fin}=\Lambda=C(\Gamma)$. 

Second, we show that if $\Gamma$ is abelian, then $|C(\Gamma)|\leq 2$. 

If $\Gamma$ is finite, then $L(\Gamma)\cong \ell^{\infty}(\{1,2,\ldots, n\})\cong \mathbb{C}^n$ where $n=|\Gamma|$. If $n\geq 3$, then let $p=(1,0,\ldots, 0)$. We see that $N:=\bC p\oplus \bC (1-p)$ is a $\Gamma$-invariant von Neumann subalgebra. Observe that $p$ and $1-p$ are minimal projections in $N$ with nonequal trace. We argue that $N\neq L(\Lambda)$ for any subgroup $\Lambda$ in $\Gamma$. Indeed, any subgroup $\Lambda$ is still finite and abelian. Therefore, $L(\Lambda)\cong \ell^{\infty}(\{1,2,\ldots, d\})$ where $d=|\Lambda|$. But, all  minimal projections in $L(\Lambda)$ have equal trace $1/d$.

If $\Gamma$ is infinite, then $L(\Gamma)\cong L^{\infty}([0, 1],\mu)$, where $\mu$ is the Lebesgue measure on $[0, 1]$. Then clearly, $L(\Gamma)$ contains uncountably many two dimensional von Neumann subalgebras, e.g., $P_t=\mathbb{C}p_t\oplus \mathbb{C}(1-p_t)$, where $p_t=\chi_{[0, t]}$ is the characteristic function on $[0, t]$, for all $0\leq t<1/2$. Note that if $P_t=L(\Lambda_t)$ for some $\Lambda_t\leq \Gamma$, then $\Lambda_t\cong \mathbb{Z}/{2\mathbb{Z}}$. In particular, $\Lambda_t$ is generated by two elements. However, $\Gamma$ contains at most countably many two-generated subgroups. This shows $L(\Gamma)$ does not have ISR property, a contradiction.

Next, we claim that $|C(\Gamma)|\leq 2$ for a general group $\Gamma$ with the ISR property. To see this, observe that any von Neumann subalgebra of $L(C(\Gamma))$ is $\Gamma$-invariant, hence we may apply the second part of the proof above to deduce $|C(\Gamma)|\leq 2$. (Indeed, if $L(\Lambda)=N$, where $N$ is defined as above, then automatically, $\Lambda\subseteq N\subseteq L(C(\Gamma))$, hence $\Lambda\subseteq C(\Gamma)$.) Therefore, we can deduce that $|\Gamma_{fin}|=|C(\Gamma)|\leq 2$.

Finally, we further assume $\Gamma$ is infinite and show $\Gamma$ is i.c.c.

Towards a contradiction, assume not. Then $C(\Gamma)=\{e, s\}$, where $s$ is an element in $\Gamma$ with order two. Clearly $p=\frac{1+\lambda(s)}{2}$ is a central projection in $L(\Gamma)$. 

Set $M=pL(\Gamma)\oplus (1-p)\mathbb{C}$. Clearly, $M$ is a $\Gamma$-invariant von Neumann subalgebra of $L(\Gamma)$, hence $M=L(H)$ for some normal subgroup $H\lhd \Gamma$.

Claim: $H=\Gamma$.

Indeed, from $1-p\in M=L(H)$, we deduce that $C(\Gamma)\subseteq H$.
Take any $g\in \Gamma\setminus C(\Gamma)$, note that $\frac{\lambda(g)+\lambda(sg)}{2}=p\lambda(g)\in M=L(H)$. Since $g\neq sg$, we deduce that $g, sg\in H$. Hence, $H=\Gamma$ and $M=pL(\Gamma)\oplus (1-p)\mathbb{C}=L(\Gamma)$. After multiplying both sides by $1-p$, we deduce that $(1-p)\mathbb{C}=(1-p)L(\Gamma)$. Thus, for any $g\in \Gamma\setminus C(\Gamma)$, we deduce that $(1-p)\lambda(g)=(1-p)c$ for some $c\in \mathbb{C}$. Clearly, this leads to a contradiction.
\end{proof}
Now, it is natural to wonder whether $\Gamma$ satisfies the ISR property under the additional assumption of being i.c.c., where $\Gamma$ belongs to the class of groups considered in \cite[Corollary 3.17]{chifan2020rigidity}. We make the following remark.

\begin{remark}
\label{remark: Chifan-Das work}
The proofs of Theorem 3.16 and Corollary 3.17 in \cite{chifan2020rigidity} actually show stronger conclusion if we further assume that the ambient group $\Gamma$ is i.c.c. The reason is recorded in the proof of the following proposition.
\end{remark}
\begin{prop}
\thlabel{obsfromCD}
Let $\Gamma$ be a group as in \cite[Corollary~3.17]{chifan2020rigidity}. Further assume that $\Gamma$ is an i.c.c. group. Then, every $\Gamma$-invariant subfactor $N\le L(\Gamma)$ is of the form $L(\Lambda)$ for some normal subgroup $\Lambda\lhd \Gamma$.
\end{prop}
\begin{proof}
Suppose that $\Gamma$ is an i.c.c. group. By considering the conjugation action of $\Gamma$ on itself, it follows from \cite[Theorem~2.5 \text{and} Theorem~3.3] {jolissaint2012examples} that the $\Gamma$-action on $L(\Gamma)$ is weakly mixing, i.e., the only finite-dimensional $\Gamma$-invariant subspace in $\ell^2(\Gamma)$ is $\bC1$ (also see \cite[Lemma 3.5.5]{Brugger18}). Thus, every finite-dimensional $\Gamma$-invariant von Neumann  subalgebra in $L(\Gamma)$ is $\bC1$. With this fact in mind, the same proof of \cite[Theorem~3.16]{chifan2020rigidity} actually shows that there is a normal subgroup $\Lambda\lhd \Gamma$ such that $N\leq L(\Lambda)\leq N\vee N'\cap L(\Lambda)$ and one of the following holds:
\begin{enumerate}
    \item $N=\bC 1$, or
    \item $\Lambda$ is infinite amenable, or
    \item $N=L(\Lambda)$.
\end{enumerate}
However, since $\Gamma$ has only trivial amenable radical, we see that item (2) does not appear. Now, we can finish the proof by taking $\Lambda=\{e\}$ when item (1) appears.
\end{proof}
However, it is not clear to us whether the factorial assumption on $N$ can be relaxed to just von Neumann subalgebras in general.
We now present examples of groups without ISR property.

\begin{example}\label{example}
$\Gamma=\bZ/{n\bZ}\oplus F_2$ does not have ISR property for any $n\geq 2$, where $F_2$ denotes the nonabelian free group generated by two elements.
This is clear by Proposition \ref{necessary_condition_for_ISR}.
\end{example}
The same example also reveals that we can not drop the torsion free assumption from \thref{invariantsubalgebras} in general.  One can also construct examples of i.c.c. groups without the ISR property.
\begin{example}
Let $A$ be an abelian group such that it contains nontrivial elements with order larger than two. Let $\Gamma=A\rtimes H$, where $H\curvearrowright A$ by group automorphisms. Moreover, assume that $\Gamma$ is i.c.c., for example, take $\Gamma=\mathbb{Z}^2\rtimes SL(2,\mathbb{Z})$ or $\mathbb{Z}\wr \mathbb{Z}=(\oplus_{\mathbb{Z}}\mathbb{Z})\rtimes \mathbb{Z}$. Now, define $\mathcal{M}$ as the von Neumann subalgebra of $L(A)$ with symmetric Fourier coefficients, i.e., $\mathcal{M}=\{x=\sum_{v\in A}\lambda_vu_v, \lambda_v=\lambda_{-v},~\forall~v\in A\}\subset L(A)$. Then, clearly $\mathcal{M}$ is a $\Gamma$-invariant subalgebra which is not of the form $L(\Lambda)$ for any $\Lambda\lhd \Gamma$.
\end{example}

\section{Proof of main results}\label{sec: proof of theorems}

In this section, we prove the theorems in the introduction.

\noindent

\begin{proof}[Proof of \thref{invariantsubalgebras}] Let $\Gamma$ be a torsion free non-amenable hyperbolic group. In addition to having cyclic centralizers for nontrivial elements, it is well known that these groups have property $P_{\text{nai}}$. The proof is now a consequence of  \thref{abstractcond} and \thref{conditionsforabstractcond}.
\end{proof}

In fact, we observe that the proof of \thref{abstractcond} can be easily generalized to deal with direct product groups. In preparation for this, we introduce a useful lemma.

\begin{lem}\thlabel{lem: kill rest terms in direct product case}
Let $n\geq 2$ and $\Gamma=\Gamma_1\times \cdots \times\Gamma_n$ be the direct product of torsion free groups $\Gamma_i$. Let $\cM\subseteq L(\Gamma)$ be a von Neumann subalgebra. Let $\mathbb{E}_{\cM}: L(\Gamma)\to \cM$ be the trace preserving conditional expectation onto $\cM$. Fix any nontrivial elements $\gamma_i\in \Gamma_i$ for $1\leq i\leq n$. Assume 
\begin{align}\label{eq: identity 1 in lemma in direct product case}
\mathbb{E}_{\cM}(\lambda(\gamma_1,\ldots, \gamma_n))=\sum_{(\epsilon_i)_i\in \{0, 1\}^n}c_{\epsilon_1\ldots\epsilon_n}\lambda(\gamma_1^{\epsilon_1},\ldots, \gamma_n^{\epsilon_n}),
\end{align}
where $c_{\epsilon_1\ldots\epsilon_n}\in \bC$ and 
\begin{align}\label{eq: identity 2 in lemma in direct product case}
\mathbb{E}_{\cM}\left(\lambda(\gamma_1^{\epsilon_1},\ldots, \gamma_n^{\epsilon_n})\right)=\theta_{\epsilon_1\ldots\epsilon_n}\lambda(\gamma_1^{\epsilon_1},\ldots, \gamma_n^{\epsilon_n}),~ \forall~(\epsilon_1,\ldots, \epsilon_n)\neq (1,1,\ldots, 1),
\end{align}
where $\theta_{\epsilon_1\ldots\epsilon_n}\in\bC$. Then $c_{\epsilon_1\ldots\epsilon_n}=0$ for all $(\epsilon_1,\ldots, \epsilon_n)\neq (1,1,\ldots, 1)$.
\end{lem}
For ease of notations in the proof below, we will ignore the $\lambda$ part, i.e., we will write $\lambda(g)$ simply as $(g)$ for all $g\in \Gamma$. Moreover, $\tau$ will denote the canonical trace on $L(\Gamma)$. It then follows that $\langle a, b\rangle=\tau(b^*a)$ for any $a, b\in L(\Gamma)$.  For the simplicity of writing, we shall drop $\mathcal{M}$ from the canonical conditional expectation $\mathbb{E}_{\M}$ and write it as $\mathbb{E}$.
\begin{proof}[Proof of Lemma \ref{lem: kill rest terms in direct product case}]

Applying $\mathbb{E}(\cdot)$ on both sides of \eqref{eq: identity 1 in lemma in direct product case}, we obtain that 
\begin{align*}
&\sum_{(\epsilon_i)_i\in \{0, 1\}^n}c_{\epsilon_1\ldots\epsilon_n}(\gamma_1^{\epsilon_1},\ldots, \gamma_n^{\epsilon_n})\\
&=\mathbb{E}(\gamma_1,\ldots, \gamma_n)\\
&=\mathbb{E}\left(\sum_{(\epsilon_i)_i\in \{0, 1\}^n}c_{\epsilon_1\ldots\epsilon_n}(\gamma_1^{\epsilon_1},\ldots, \gamma_n^{\epsilon_n})\right)\\
&=\sum_{(\epsilon_i)_i\in \{0, 1\}^n}c_{\epsilon_1\ldots\epsilon_n}\E(\gamma_1^{\epsilon_1},\ldots, \gamma_n^{\epsilon_n})\\
&\overset{\eqref{eq: identity 2 in lemma in direct product case}}{=}c_{11\ldots 1}\sum_{(\epsilon_i)_i\in \{0, 1\}^n}c_{\epsilon_1\ldots\epsilon_n}(\gamma_1^{\epsilon_1},\ldots, \gamma_n^{\epsilon_n})
+\sum_{(\epsilon_1,\ldots, \epsilon_n)\neq (1,1,\ldots, 1)}c_{\epsilon_1\ldots\epsilon_n}\theta_{\epsilon_1\ldots\epsilon_n}(\gamma_1^{\epsilon_1},\ldots, \gamma_n^{\epsilon_n}).
\end{align*}
By comparing coefficients on both sides, we deduce that
\begin{align}\label{eq: system of eqs I}
\begin{cases}
c_{\epsilon_1\ldots\epsilon_n}&=c_{11\ldots 1}c_{\epsilon_1\ldots\epsilon_n}+c_{\epsilon_1\ldots\epsilon_n}\theta_{\epsilon_1\ldots\epsilon_n},~\forall~(\epsilon_1,\ldots, \epsilon_n)\neq (1,1,\ldots, 1),\\
c_{11\ldots 1}&=c_{11\ldots 1}^2.
\end{cases}
\end{align}
Hence, $c_{11\ldots 1}$ is either $0$ or $1$. We will deal with these two cases separately.
Suppose that $c_{11\ldots 1}=0$. Then, plugging it into \eqref{eq: system of eqs I}, we get that $c_{\epsilon_1\ldots\epsilon_n}(1-\theta_{\epsilon_1\ldots\epsilon_n})=0$. We now claim that $\forall~(\epsilon_1,\ldots, \epsilon_n)\neq (1,1,\ldots, 1)$, 
\[\overline{\theta_{\epsilon_1\ldots\epsilon_n}}c_{\epsilon_1\ldots\epsilon_n}=\left\langle (\gamma_1,\ldots, \gamma_n)-\E(\gamma_1,\ldots, \gamma_n), \E(\gamma_1^{\epsilon_1},\ldots,\gamma_n^{\epsilon_n})\right\rangle=0.\]
Indeed, first using the fact that for all $x\in \mathcal{M}$, $x-\mathbb{E}(x)$ is perpendicular to $\M$ with respect to the inner product induced by the trace $\tau$, we obtain that
\begin{align*}
\left\langle (\gamma_1,\ldots, \gamma_n)-\E(\gamma_1,\ldots, \gamma_n), \E(\gamma_1^{\epsilon_1},\ldots,\gamma_n^{\epsilon_n})\right\rangle=0.
\end{align*}
On the other hand, a plain calculation yields that
\begin{align*}
&\left\langle (\gamma_1,\ldots, \gamma_n)-\E(\gamma_1,\ldots, \gamma_n), \E(\gamma_1^{\epsilon_1},\ldots,\gamma_n^{\epsilon_n})\right\rangle\\&=\tau\left(\left(\E(\gamma_1^{\epsilon_1},\ldots,\gamma_n^{\epsilon_n})\right)^*\left((\gamma_1,\ldots, \gamma_n)-\E(\gamma_1,\ldots, \gamma_n)\right)\right)\\&=\tau\left(\left(\overline{\theta_{\epsilon_1\ldots\epsilon_n}}(\gamma_1^{\epsilon_1},\ldots,\gamma_n^{\epsilon_n})^*\right)\left((\gamma_1,\ldots, \gamma_n)-E(\gamma_1,\ldots, \gamma_n)\right)\right)\\&=-\overline{\theta_{\epsilon_1\ldots\epsilon_n}}c_{\epsilon_1\ldots\epsilon_n},~\forall~(\epsilon_1,\ldots, \epsilon_n)\neq (1,1,\ldots, 1).
\end{align*}
Therefore, we deduce that $c_{\epsilon_1\ldots\epsilon_n}\overline{\theta_{\epsilon_1\ldots\epsilon_n}}=0$. Hence, we get that $c_{\epsilon_1\ldots\epsilon_n}=0$ for all $(\epsilon_1,\ldots, \epsilon_n)\neq (1,1,\ldots, 1)$, which finishes the proof.\\
Now, suppose that $c_{11\ldots 1}=1$. Plugging it into \eqref{eq: system of eqs I}, we get that $c_{\epsilon_1\ldots\epsilon_n}\theta_{\epsilon_1\ldots\epsilon_n}=0$ for all $(\epsilon_1,\ldots, \epsilon_n)\neq (1,1,\ldots, 1)$. By applying $\mathbb{E}(\cdot)$ on both sides of \eqref{eq: identity 2 in lemma in direct product case} we see that $\theta_{\epsilon_1\ldots\epsilon_n}\in \{0, 1\}$ for all $(\epsilon_1,\ldots, \epsilon_n)\neq (1,1,\ldots, 1)$.

If $\theta_{\epsilon_1\ldots\epsilon_n}=0$, then we can argue similarly as above to obtain that \[\overline{c_{\epsilon_1\ldots\epsilon_n}}=\langle (\gamma_1^{\epsilon_1},\ldots, \gamma_n^{\epsilon_n})-E(\gamma_1^{\epsilon_1},\ldots,\gamma_n^{\epsilon_n}), E(\gamma_1,\ldots, \gamma_n) \rangle=0.\]

If $\theta_{\epsilon_1\ldots\epsilon_n}=1$, then we directly get that $c_{\epsilon_1\ldots\epsilon_n}=c_{\epsilon_1\ldots\epsilon_n}\theta_{\epsilon_1\ldots\epsilon_n}=0$.

Either way, it follows that $c_{\epsilon_1\ldots\epsilon_n}=0$ for all $(\epsilon_1,\ldots, \epsilon_n)\neq (1,1,\ldots, 1)$. This completes the proof.
\end{proof}
We now proceed to prove ISR property for a finite direct product of groups which satisfy the assumptions of \thref{abstractcond}.
We remark that a group having Property $P_{\text{nai}}$ is necessarily i.c.c..
\begin{prop}\thlabel{prop: direct product of groups with the two conditions have ISR property}
Let $n\geq 2$ and $\Gamma_i$ be groups as in  \thref{{abstractcond}} for all $1\leq i\leq n$. Then $\Gamma:=\Gamma_1\times \cdots  \times \Gamma_n$ satisfies ISR property.
\end{prop}
\noindent
We now proceed to prove the above for $n=2$. The general case would then follow by an easy induction argument. Let $(g,h)\in \Gamma_1\times\Gamma_2$ be a nontrivial element. As explained in the proof of \thref{{abstractcond}}, our goal is to show $\lambda(g, h)^{-1}\mathbb{E}_{\cM}(\lambda(g, h))\in \bC$. There are now three different cases to consider.\\ 
\textit{Case 1: $h=e$ and $g\neq e$.}\\
\noindent
Write $g=g_1^n$ for some primitive $g_1\in \G_1$ and some $n\geq 1$. Using the fact that $\Gamma_2$ is i.c.c, we notice that 
$\E_{\mathcal{M}}(\lambda(g, e))\in L(\langle g_1\rangle\times \G_2)'\cap L(\G_1\times \G_2)=L(\langle g_1\rangle\times \{e\})$. Therefore, we can write $$\E_{\mathcal{M}}(\lambda(g, e))=\sum_{k\in \mathbb{Z}}c_k\lambda(g_1^k, e).$$
Pick $s\in G_1$ such that $sg_1s^{-1}$ is free from $g_1$. Now, arguing similarly as in the proof of \thref{{abstractcond}}, we conclude that $\lambda(g, e)^{-1}\mathbb{E}_{\cM}(g, e)\in \bC$.\\
\textit{Case 2: $g=e$ and $h\ne e$.}\\
\noindent
By symmetry, this follows as in Case 1.\\ 
\textit{Case 3: $g\neq e\neq h$.}\\
\noindent
We write $g=g_1^n$ and $h=h_1^m$ for two primitive elements $g_1\in \G_1$ and $h_1\in \G_2$ and some $n, m\geq 1$. Arguing similarly as before, we obtain that $E_{\mathcal{M}}(\lambda(g, h))\in L(\langle g_1\rangle\times \langle h_1\rangle)'\cap L(\G_1\times \G_2)=L(\langle g_1\rangle\times \langle h_1\rangle)$. Hence, we can write $$\mathbb{E}_{\cM}(\lambda(g, h))=\sum_{i, j\in\mathbb{Z}}c_{ij}\lambda(g_1^i, h_1^j).$$ Pick $s\in \G_1$(resp. $t\in \G_2$) such that $sg_1s^{-1}$ is free from $g_1$ (resp. $tg_2t^{-1}$ is free from $g_2$). 
Therefore, we see that
\begin{align*}
\mathbb{E}_{\cM}(\lambda(g, h))\cdot \mathbb{E}_{\cM}(\lambda((s, t)(g, h)(s, t)^{-1}))
=\sum_{i, j, i', j'\in\mathbb{Z}}c_{ij}c_{i'j'}\lambda(g_1^isg_1^{i'}s^{-1}, h_1^jth_1^{j'}t^{-1}).\end{align*}
\begin{align*}
\mathbb{E}_{\cM}\big(\lambda(g, h)\cdot \mathbb{E}_{\cM}(\lambda((s, t)(g, h)(s, t)^{-1}))\big)
=\sum_{i'', j''\in\mathbb{Z}}c_{i''j''}\mathbb{E}_{\cM}(\lambda(gsg_1^{i''}s^{-1}, hth_1^{j''}t^{-1})).
\end{align*}
Following the same strategy as in \thref{{abstractcond}}, we obtain the following claim.\\
\textit{Claim 1: $(g_1^isg_1^{i'}s^{-1}, h_1^jth_1^{j'}t^{-1})\in G$ is uniquely determined by the tuple $(i, i', j, j')$.}\\
This in particular implies that if $c_{ij}c_{i'j'}\neq 0$, then \[(g_1^isg_1^{i'}s^{-1}, h_1^jth_1^{j'}t^{-1})\in supp(\mathbb{E}_{\cM}(\lambda(g, h))\cdot \mathbb{E}_{\cM}(\lambda((s, t)(g, h)(s, t)^{-1}))).\]
It is evident that if we assume either $i\neq n\wedge i'\neq 0$ or $j\neq m\wedge j'\neq 0$, then \[(g_1^isg_1^{i'}s^{-1}, h_1^jth_1^{j'}t^{-1})\not\in supp(\mathbb{E}_{\cM}\big(\lambda(g, h)\cdot \mathbb{E}_{\cM}(\lambda((s, t)(g, h)(s, t)^{-1}))\big)).\]
As a result, we obtain the following claim.\\
\textit{Claim 2: Assume $i\neq n\wedge i'\neq 0$ (resp. $j\neq m\wedge j'\neq 0$), then $c_{ij}c_{i'j'}=0$ for any $j, j'$ (resp. for any $i, i'$).}\\ 
Now we prove the following.\\
\textit{Claim 3: Assume that $i\not\in\{ n, 0\}$ (resp. $j\not\in\{ m, 0\}$). Then, $c_{ij}=0$ for all $j$ (resp. for all $i$).}\\
Towards a contradiction, suppose otherwise. Namely, assume that for some $i\not\in\{ n, 0\}$ and some $j$, $c_{ij}\neq 0$. Then for any $i'\neq 0$ and all $j'$, we deduce that $c_{i'j'}=0$ by Claim 2. By taking $(i', j')=(i, j)$, we obtain a contradiction. The case $j\neq m, 0$ can be proved similarly.\\

\noindent
Since $\tau(\lambda(g, h))=0$, we see that $c_{00}=0$. Combining this along with Claim 3, we obtain that 
\[\mathbb{E}_{\cM}(\lambda(g, h))=c_{n0}\lambda(g, e)+c_{0m}\lambda(e, h)+c_{nm}\lambda(g, h).\]
We are now left to show that $c_{n0}=c_{0m}=0$. Using Case 1 and Case 2, we may write $\mathbb{E}_{\cM}(\lambda(g, e))=c_n\lambda(g, e)$ and $\mathbb{E}_{\cM}(\lambda(e, h))=c_m\lambda(e, h)$. By taking $n=2$ and $(\gamma_1,\gamma_2)=(g, h)$ in \thref{lem: kill rest terms in direct product case}, we finish the proof for $n=2$.

\begin{proof}[Proof of \thref{prop: direct product of groups with the two conditions have ISR property}]
We now sketch the proof for general $n\geq 2$, which is based on induction on the number of $\G_i's$.
For ease of notations, we write $\mathbb{E}(g)$ instead of $\mathbb{E}_{\cM}(\lambda(g))$ for all $g\in \G$.
\noindent
Assume that the proposition holds when $\sharp~ \Gamma_i's<n$. Now, fix any $(g_1,\ldots, g_n)\in \Gamma$. Write $k=\sharp\{i:~g_i=e\}$. We may assume $k<n$.

If $k>0$, i.e., $\exists~i\in \{1,\ldots, n\}$ s.t. $g_i=e$, then it is not hard to see $\E((g_1,\ldots, g_n))\in\bC (g_1,\ldots, g_n)$ by essentially the same proof when the number of $\Gamma_i's$ is $k$. Therefore, without any loss of generality, we may assume $k=0$, i.e., all $g_i'$s are nontrivial. Write $g_i=\underline{g_i}^{\ell_i}$, where $\underline{g_i}$ are primitive elements in $\G_i$ and $\ell_i\geq 1$ for all $1\leq i\leq n$. Pick any $s_i\in G_i$ such that $\underline{g_i}$ is free from $s_i\underline{g_i}s_i^{-1}$ for all $1\leq i\leq n$.
Hence, we can write 
\begin{align*}
\E((g_1,\ldots, g_n))=\sum_{i_1,\ldots, i_n\in\mathbb{Z}}c_{i_1,\ldots, i_n}(\underline{g_1}^{i_1},\ldots, \underline{g_n}^{i_n}),\end{align*}
\begin{align*}
\E\big((s_1,\ldots, s_n)(g_1,\ldots, g_n)(s_1,\ldots, s_n)^{-1}\big)=\sum_{i_1',\ldots, i_n'\in\mathbb{Z}}c_{i_1',\ldots, i_n'}(s_1\underline{g_1}^{i_1}s_1^{-1},\ldots, s_n\underline{g_n}^{i_n}s_n^{-1}).
\end{align*}
Therefore, we have
\begin{align*}
&\sum_{i_1,\ldots, i_n\in\mathbb{Z}}\sum_{i_1',\ldots, i_n'\in\mathbb{Z}}c_{i_1,\ldots, i_n}c_{i_1',\ldots, i_n'}(\underline{g_1}^{i_1}s_1\underline{g_1}^{i_1'}s_1^{-1},\ldots, \underline{g_n}^{i_n}s_1\underline{g_n}^{i_n'}s_n^{-1})\\&=
\E((g_1,\ldots, g_n))\cdot \E((s_1g_1s_1^{-1},\ldots, s_ng_ns_n^{-1}))\\
&=
\E\big((g_1,\ldots, g_n)\cdot \E((s_1g_1s_1^{-1},\ldots, s_ng_ns_n^{-1}))\big)\\&
=\sum_{i_1'',\ldots, i_n''\in\mathbb{Z}}c_{i_1'',\ldots, i_n''}\E((g_1s_1\underline{g_1}^{i_1''}s_1^{-1},\ldots, g_ns_n\underline{g_n}^{i_n''}s_n^{-1})).
\end{align*}
Arguing similarly as in the proofs of \thref{abstractcond} and the case for $n=2$, the following claims follow.

\textit{Claim 1: $(\underline{g_1}^{i_1}s_1\underline{g_1}^{i_1'}s_1^{-1},\ldots, \underline{g_n}^{i_n}s_1\underline{g_n}^{i_n'}s_n^{-1})$ is uniquely determined by the tuple $(i_1,\ldots, i_n, i_1',\ldots, i_n')$.}

\textit{Claim 2: For any $j\in \{1,\ldots, n\}$, if $i_j\neq \ell_j$ and $i_j'\neq 0$, then $c_{i_1,\ldots, i_n}c_{i_1',\dots, i_n'}=0$ for all $i_t$, $i'_t$ and all $t\in \{1,\ldots, n\}\setminus \{j\}$.}

\textit{Claim 3: For any $j\in \{1,\ldots, n\}$, if $i_j\not\in\{ \ell_j, 0\}$, then $c_{i_1,\ldots, i_n}=0$ for all $i_1,\ldots,\widehat{i_j},\ldots, i_n$, where $\widehat{i_j}$ means $i_j$ does not appear.}

Using Claim 3, we can conclude that
\begin{align*}
\E((g_1,\ldots, g_n))=\sum_{(\epsilon_i)_i\in \{0, 1\}^n}c_{\epsilon_1\ldots\epsilon_n}(g_1^{\epsilon_1},\ldots, g_n^{\epsilon_n}).
\end{align*}
We can now appeal to Lemma \ref{lem: kill rest terms in direct product case} in order to conclude that 
$E((g_1,\ldots, g_n))\in \bC(g_1,\ldots, g_n)$.

This finishes the whole proof.
\end{proof}

We do not know whether Theorem 3.16 and Corollary 3.17 in \cite{chifan2020rigidity} can be generalized to deal with direct product groups. The following question is natural.

\begin{question}
Let $G$ and $H$ be two infinite groups satisfying ISR property. Does $G\times H$ also satisfy this property?
\end{question}

Now, let $\Gamma$ be a group as in \cite[Theorem~4.1]{PetersonThom2011}. We can write $\Gamma$ as a disjoint union $\Gamma=\{e\}\cup\bigcup_{i\in I}\dot{\Gamma_i}$, where $\{\Gamma_i:~i\in I\}$ is a family of subgroups of $\Gamma$ and $\dot{\Gamma_i}=\Gamma_i\setminus \{e\}$. Moreover, $\Gamma_i$ is free from $\Gamma_j$ for $i\neq j$. Also, it is assumed that $\beta_1^{(2)}(\Gamma_i)=0$ for all $i\in I$. In particular, since $\beta^{(2)}_1(\Gamma)>0$, we may write $\{1, 2\}\subseteq I$ and $\Gamma_1$ and $\Gamma_2$ are nontrivial.

\begin{proof}[Proof of \thref{invariantsubalgebras_2}]
Let $e\ne g\in \Gamma$ be given.
It follows from the above observation that
we can assume $g\in \Gamma_1$ without any loss of generality. Since $\Gamma_j$ is free from $\Gamma_1$ for any $j\neq 1$, we get that $C_{\Gamma}(g^i)\subseteq \Gamma_1$ for all $i\geq 1$. Thus, by an application of \thref{malnormal}, we see that $L(\langle g\rangle)'\cap L(\Gamma)\subseteq L(\Gamma_1)$. 
Hence, $\mathbb{E}_{\cM}(\lambda(g))\in L(\langle g\rangle)'\cap L(\Gamma)\subseteq L(\Gamma_1)$. We  write 
\[\mathbb{E}_{\cM}(\lambda(g))=\sum_{g_1\in \Gamma_1}c_{g_1}\lambda(g_1),~\mbox{where $c_{g_1}\in \bC$.}\]
Now, we pick a nontrivial $s\in \Gamma_2$. We have
\begin{equation}
\label{fourexp}
 \mathbb{E}_{\cM}(\lambda(sgs^{-1}))=\lambda(s)\mathbb{E}_{\cM}(\lambda(g))\lambda(s)^{-1}=\sum_{g_2\in \Gamma_1}c_{g_2}\lambda(sg_2s^{-1}).   
\end{equation}
We  now observe that
if $g_1sg_2s^{-1}=g_1'sg_2's^{-1}$, then
$$g_1'^{-1}g_1=sg_2'g_2^{-1}s^{-1}\in\Gamma_1\cap s\Gamma_1s^{-1}=\{e\}, \text{ i.e.,} ~(g_1, g_2)=(g_1', g_2').$$
Therefore, it follows that $\forall ~g_1, g_1', g_2, g_2'\in \Gamma_1$, $$g_1sg_2s^{-1}=g_1'sg_2's^{-1} \iff (g_1, g_2)=(g_1', g_2').$$
As a consequence, we obtain that \[supp(\mathbb{E}_{\cM}(\lambda(g))\cdot \mathbb{E}_{\cM}(\lambda(sgs^{-1})))=supp(\mathbb{E}_{\cM}(\lambda(g)))\cdot supp(\mathbb{E}_{\cM}(\lambda(sgs^{-1}))),\] so we may still simply write 
\begin{align*}
\mathbb{E}_{\cM}(\lambda(g))\cdot \mathbb{E}_{\cM}(\lambda(sgs^{-1}))
=\sum_{g_1, g_2\in \Gamma_1}c_{g_1}c_{g_2}\lambda(g_1sg_2s^{-1}).
\end{align*}
On the other hand, 
\begin{align*}
\mathbb{E}_{\cM}(\lambda(g)\cdot\mathbb{E}_{\cM}(\lambda(sgs^{-1})))
=\sum_{g_3\in \Gamma_1}c_{g_3}\mathbb{E}_{\cM}(\lambda(gsg_3s^{-1})).
\end{align*}
Note that \thref{control_support_in_relative_commutant} implies that
\[supp(\mathbb{E}_{\cM}(\lambda(g)\cdot\mathbb{E}_{\cM}(\lambda(sgs^{-1})))\subseteq \cup_{g_3\in\Gamma_1}\cup_{i\geq 1}C_{\Gamma}((gsg_3s^{-1})^i).\]
For $g_1\neq g$ and $g_2\neq e$, we claim the following.\\
\noindent
\textit{Claim: 
$g_1sg_2s^{-1}\not\in supp(\mathbb{E}_{\cM}(\lambda(g)\cdot\mathbb{E}_{\cM}(\lambda(sgs^{-1})))$.}\\
\noindent
\textit{Proof of Claim:}
It suffices to show that for any $g_3\in \Gamma_1$ and any $i\geq 1$, we have $g_1sg_2s^{-1}\not\in C_{\Gamma}((gsg_3s^{-1})^i)$.
Towards a contradiction, suppose that this is not the case. Then, 
\[g_1sg_2s^{-1}(gsg_3s^{-1})^i=(gsg_3s^{-1})^ig_1sg_2s^{-1}.\]
First, observe that $g_3\neq e$. If not, then the above identity implies that
$g_1sg_2s^{-1}g^i=g^ig_1sg_2s^{-1}$. We note that $g_2\neq e$, and $s$ is free from $\dot{\Gamma_1}$ which contains $\{g_1, g_2, g\}$. Hence, by working inside $\Gamma_1*\langle s\rangle$, we see that the ending letter on the left hand side is $g\in \Gamma_1$. However, this is different from the ending letter $s^{-1}\in \langle s \rangle$ on the right hand.  This is a contradiction.

Second, observe that we may assume that $g_1\neq e$ for the proof. Suppose otherwise, i.e., $g_1=e$. Then, we have $sg_2s^{-1}(gsg_3s^{-1})^i=(gsg_3s^{-1})^isg_2s^{-1}$. Since $g, g_2, g_3\neq e$ and $s$ is free from $\Gamma_1$ which contains $\{g_2, g_3, g\}$, the starting letter on the left hand side is $s\in \Gamma_2$. But, the starting letter on the right hand side is $g\in \Gamma_1$. This is again a contradiction. Therefore, we can assume that $g_1=e$.

Now, since $g_1\neq g$ and $e\not\in \{g_1, g_2, g\}$, we obtain a contradiction by comparing the initial letters on both sides of the original identity.
This completes the proof of the claim.

Using the above \textit{Claim} along with the equality $$\mathbb{E}_{\cM}(\lambda(g))\mathbb{E}_{\cM}(\lambda(sgs^{-1}))=\mathbb{E}_{\cM}(\lambda(g)\mathbb{E}_{\cM}(\lambda(gsg^{-1}))),$$ we deduce that 
$$c_{g_1}c_{g_2}=0,\forall g_1\neq g \text{ and } g_2\neq e.$$ 

If for some $g_1\neq g$, we have $c_{g_1}\neq 0$, then it follows that $c_{g_2}=0$ for all $g_2\neq e$. Therefore, from equation~(\ref{fourexp}), we obtain that
\[\mathbb{E}_{\cM}(\lambda(sgs^{-1}))=c_e\lambda(e).\]
Applying the canonical trace $\tau$ on both sides and using the fact that $\E_{\M}$ is trace preserving, we see that
\[0=\tau(\lambda(sgs^{-1})=\tau\left(\mathbb{E}_{\cM}(\lambda(sgs^{-1}))\right)=\tau\left(c_e\lambda(e)\right)=c_e.\]
It then follows that $\mathbb{E}_{\cM}(\lambda(sgs^{-1}))=0$. Hence, $c_{g_1}=0$, which cannot happen. Therefore, for all $g_1\neq g$, $c_{g_1}=0$, i.e., $\lambda(g)^{-1}\mathbb{E}_{\cM}(\lambda(g))\in \bC$. This completes the proof.
\end{proof}
\noindent
\thref{invariantsubalgebras_3} is all that remains to be shown. If all $\Gamma_i$'s are groups satisfying the assumptions of \thref{invariantsubalgebras}, then the proof follows by applying  \thref{prop: direct product of groups with the two conditions have ISR property} along with \thref{conditionsforabstractcond}. Therefore, we assume that all $\Gamma_i$'s are as in \thref{invariantsubalgebras_2}. We  also note that all the $\Gamma_i's$ are i.c.c. groups.

We will only present the proof for $n=2$ since the general case is similar to that of \thref{prop: direct product of groups with the two conditions have ISR property}.

\begin{proof}[Proof of \thref{invariantsubalgebras_3}]
For $j=1$ and $2$, let $\Gamma_j=\{e\}\cup\bigcup_i\dot{\Gamma_{ji}}$ for $j=1,2$ be the decomposition given in \cite[Theorem ~4.1]{PetersonThom2011}. Fix any nontrivial $(g_1, g_2)\in \Gamma$. We may assume $g_j\in \Gamma_{j1}$, for $j=1,2$. Our goal is to show that $\mathbb{E}_{\cM}(\lambda(g_1, g_2))\in \bC\lambda(g_1, g_2)$. As before, we have three cases to consider.\\
\noindent
\textit{Case 1: $g_1=e$ and $g_2\neq e$.}\\
Since $\mathbb{E}_{\cM}(\lambda(e, g_2))\in L(\Gamma_1\times \langle g_2\rangle)'\cap L(\Gamma)\subseteq L(\{e\}\times \Gamma_{21})$, we may write 
\[\mathbb{E}_{\cM}(e, g_2)=\sum_{g_{2j}\in \Gamma_{21}}c_{g_{2j}}\lambda(e, g_{2j})~\mbox{where}~c_{g_{2j}}\in \bC.\]
For any nontrivial $t\in \Gamma_{22}$, we have
\[\mathbb{E}_{\cM}(\lambda(e, tg_2t^{-1}))=\sum_{g_{2j}\in \Gamma_{21}}c_{g_{2j}}\lambda(e, tg_{2j}t^{-1}).\]
Arguing similarly as in the proof of \thref{invariantsubalgebras_2}, we conclude that $\mathbb{E}_{\cM}(\lambda(e, g_2))\in \bC\lambda(e, g_2)$.\\
\textit{Case 2: $g_2=e$ and $g_1\neq e$.}\\
The proof is similar to \textit{Case 1}. And, we deduce that $\mathbb{E}_{\cM}(\lambda(g_1, e))\in\bC\lambda(g_1, e)$.\\
\textit{Case 3: $g_1\neq e\neq g_2$.}\\
\noindent
Since $\mathbb{E}_{\cM}(\lambda(g_1, g_2))\in L(\langle g_1\rangle\times\langle g_2\rangle)'\cap L(\Gamma)\subseteq L(\Gamma_{11}\times \Gamma_{21})$, we can write 
\[\mathbb{E}_{\cM}(\lambda(g_1, g_2))=\sum_{i,j\in\mathbb{Z}}c_{ij}\lambda(g_{1i}, g_{2j}).\]
Note that in this case, $g_{1i}\in \Gamma_{11}$ and $g_{2j}\in \Gamma_{21}$ for all $i, j$.
For any nontrivial $s\in \Gamma_{12}$ and nontrivial $t\in\Gamma_{22}$, we have
\[\mathbb{E}_{\cM}(\lambda(sg_1s^{-1}, tg_2t^{-1}))=\sum_{i', j'\in\mathbb{Z}}c_{i'j'}\lambda(sg_{1i'}s^{-1}, tg_{2j'}t^{-1}).\]
The following claims follow exactly in the same way as in the proof of \thref{prop: direct product of groups with the two conditions have ISR property}.\\
\textit{Claim 1: $(g_{1i}sg_{1i'}s^{-1}, g_{2j}tg_{2j'}t^{-1})$ is uniquely determined by the tuple $(i, i', j, j')$.} 
Therefore,
\begin{align*}
\mathbb{E}_{\cM}(\lambda(g_1, g_2))\cdot \mathbb{E}_{\cM}(\lambda(sg_1s^{-1}, tg_2t^{-1}))=
\sum_{i, j, i', j'\in\mathbb{Z}}c_{ij}c_{i'j'}\lambda(g_{1i}sg_{1i'}s^{-1}, g_{2j}tg_{2j'}t^{-1}).
\end{align*}
Note that
\begin{align*}
\mathbb{E}_{\cM}(\lambda(g_1, g_2)\cdot \mathbb{E}_{\cM}(\lambda(sg_1s^{-1}, tg_2t^{-1})))=
\sum_{i'', j''\in\mathbb{Z}}c_{i''j''}\mathbb{E}_{\cM}(\lambda(g_1sg_{1i''}s^{-1}, g_2tg_{2j''}t^{-1})).
\end{align*}
\textit{Claim 2: Assume $g_{1i}\neq g_1\wedge g_{1i'}\neq e$ (resp. $g_{2j}\neq g_2\wedge g_{2j'}\neq e$), then $c_{ij}c_{i'j'}=0$ for all $j, j'$ (resp. for all $i, i'$).}\\ 
\textit{Claim 3: Assume $g_{1i}\neq g_1, e$ (resp. $g_{2j}\neq g_2, e$), then $c_{ij}=0$ for all $j$ (resp. for all $i$).}
Combining \textit{Claim 3} along with the fact that $\tau(\lambda(g_1, g_2))=0$, we obtain that
\[\mathbb{E}_{\cM}(\lambda(g_1, g_2))=c_{g_1, e}\lambda(g_1, e)+c_{e, g_2}\lambda(e, g_2)+c_{g_1, g_2}\lambda(g_1, g_2).\]
We can now apply Lemma \ref{lem: kill rest terms in direct product case} to conclude that $c_{g_1, e}=0=c_{e, g_2}$, i.e., \[\mathbb{E}_{\cM}(\lambda(g_1, g_2))\in \bC\lambda(g_1, g_2).\]
This finishes the proof of $n=2$.

The general case follows by an induction argument. The proof is similar to that of the induction argument illustrated in \thref{prop: direct product of groups with the two conditions have ISR property} and is left to the reader.
\end{proof}
For the groups $\Gamma$ considered in the above theorems, it follows that $L(\Gamma)$ does not admit any $\Gamma$-invariant Cartan subalgebras.
To further elaborate on it, let $(\M,\tau)$ be a tracial von Neumann algebra, i.e., a von Neumann algebra $\M$ equipped with a faithful normal tracial state $\tau : \M\rightarrow\C$.
Let $\mathcal{A}\subset\M$ be a von Neumann subalgebra. Let $\mathcal{U}(\mathcal{A})$ denote the group of unitary elements of $\mathcal{A}$. Then, $\mathcal{A}$ is called a Cartan subalgebra if 
\begin{enumerate}[label=(\alph*)]
\item $\mathcal{A}$ is the maximal abelian $*$-subalgebra inside $\M$, and,
\item the normalizer $\mathcal{N}_\M(\mathcal{A})=\{u\in \mathcal{U}(\mathcal{A})|~u\mathcal{A}u^*=\mathcal{A}\}$ generates $\M$ as a von Neumann algebra.
\end{enumerate}

\begin{cor}
\thlabel{noinvcartan}
Let $\Gamma$ be any group as in \thref{invariantsubalgebras_2} or \thref{invariantsubalgebras_3}. Then $L(\Gamma)$ does not admit any Cartan subalgebra $\mathcal{A}$ with $\lambda(\Gamma)\subseteq \mathcal{N}_{L(\Gamma)}(\mathcal{A})$.
\begin{proof}
Note that $\Gamma$ is i.c.c.. Assume such $\mathcal{A}$ exists, then $\mathcal{A}=L(\Lambda)$ for some normal subgroup $\Lambda\lhd \Gamma$. Clearly $\Lambda$ is infinite and abelian.

First, let $\Gamma$ be a group as in 
\thref{invariantsubalgebras_2}. To get a contradiction, it suffices to show that $\Gamma$ has trivial amenable radical, i.e., the maximal amenable normal subgroup of $\Gamma$ is trivial. Since every infinite amenable group has zero first $L^2$-Betti number, it can not be a normal subgroup inside $\Gamma$ with positive first $L^2$-Betti number (see e.g.,  \cite[Corollary 5.14]{PetersonThom2011}). Similarly, it is well-known that torsion-free non-amenable hyperbolic groups have trivial amenable radical (see e.g., \cite[Corollary 3]{arzhantseva2006relatively}).

Now, let $\Gamma$ be a group as in \thref{invariantsubalgebras_3}, say $\Gamma=\Gamma_1\times \cdots\times \Gamma_n$ for some $n\geq 2$. Denote by $\pi_i: \Gamma\to \Gamma_i$ the natural projection onto the $i$-th coordinate. Then, $\pi_i(\Lambda)$ is a normal abelian subgroup of $\Gamma_i$, hence $\pi_i(\Lambda)$ is trivial by the above proof. Thus $\Lambda\subseteq \prod_{i=1}^n\pi_i(\Lambda)$ is trivial. 
\end{proof}
\end{cor}
We remark that when $\Gamma$ is a non-elementary hyperbolic group, it follows from \cite[Theorem~4.1]{CS} that $L(\Gamma)$ does not have a Cartan subalgebra. 

\section{ISR property for a group with torsion elements}\label{sec: remaining questions}
%We collect more remaining questions which are worth studying in our opinion.

Motivated by \thref{invariantsubalgebras}, \thref{invariantsubalgebras_2}, \cite[Corollary 3.17]{chifan2020rigidity} and Example \ref{example}, it is natural to expect a positive answer to the following question.
\begin{question}\label{question: general case with trivial amenable radical}
Can we extend \thref{invariantsubalgebras} to all hyperbolic groups, more generally, all acylindrically hyperbolic groups, with trivial amenable radical?
\end{question}

Using our technique, it seems plausible to show that for any $g\in \Gamma$ with infinite order, we have $\lambda(g)^{-1}\mathbb{E}_{\cM}(\lambda(g))\in \bC$. We denote by $\Lambda$ the subgroup consisting of all those elements which have non-zero image under $\mathbb{E}_{\M}$, i.e.,  $\Lambda=\langle g\in \Gamma:~\lambda(g^{-1})\mathbb{E}_{\cM}(\lambda(g))\neq 0,~\mbox{$g$ has infinite order}\rangle$. Clearly, $\Lambda\lhd \Gamma$ and $L(\Lambda)\subseteq \cM$. By \cite[Corollary 3.8]{chifan2020rigidity}, we are left to show that $\Lambda$ is nontrivial. Suppose $\Lambda$ is trivial, namely, $\mathbb{E}_{\cM}(\lambda(g))=0$ for all $g\in \Gamma$ with infinite order. In other words, for any $x\in\cM$, every group element appearing in the support set of $x$ has finite order. We want to claim that this is impossible unless $\mathcal{M}=\bC$. We do not know how to show this in general. The difficulty is that our strategy relies heavily on  the Fourier expansion of $\mathbb{E}_{\cM}(\lambda(g))$, which seems hard to control directly for a torsion element $g$. 

Nevertheless, it is still possible to do some algebraic manipulations in some special cases. Below, we give an example of a hyperbolic group which has torsion elements but still satisfies the ISR property.

\begin{prop}
$\Gamma=\bZ*\frac{\bZ}{2\bZ}$ satisfies the ISR property.
\end{prop}
\begin{proof}
Let $\cM$ be a $\Gamma$-invariant von Neumann subalgebra. Our goal is to show $\lambda(g^{-1})\mathbb{E}_{\cM}(\lambda(g))\in \bC$ for all $g\in \Gamma$. 

Following the above strategy, we first show that this holds for any $g\in \Gamma$ with infinite order.

Write $\Gamma=\bZ*\frac{\bZ}{2\bZ}=\langle t\rangle *\langle s\rangle$. Observe that $\Gamma\cong F_2\rtimes \frac{\bZ}{2\bZ}=\langle t, sts\rangle\rtimes\langle s\rangle$. Thus, the set of nontrivial torsion elements in $\Gamma$ is $Tor(\Gamma)=\{\gamma s\gamma^{-1}:~\gamma\in \Gamma\}$, i.e., the set of conjugacy class of $s$. It is also clear that the centralizer $C(s)=\langle s\rangle$ since $\frac{\bZ}{2\bZ}\curvearrowright F_2$ has only the trivial element as a fixed point.

Let $g\in \Gamma$ be any element with infinite order. 

\noindent
Claim 1:  $L(\langle g\rangle)'\cap L(\Gamma)=L(\langle g_0\rangle)$ for some $g_0\in \Gamma$ with $g\in \langle g_0\rangle$.\\
\textit{Proof of Claim 1:}
In the light of \thref{control_support_in_relative_commutant}, it suffices to show that the subgroup $\cup_{i\geq 1}C(g^i)$ is infinite cyclic. Since $\Gamma$ is hyperbolic, it follows from \cite[Lemma~6.5]{dahmani} that $\cup_{i\geq 1}C(g^i)$ is infinite and virtually cyclic. Moreover, observe that $\cup_{i\geq 1}C(g^i)\cap Tor(\Gamma)=\emptyset$. Indeed, assume that $\gamma s\gamma^{-1}\in C(g^i)$ for some $i\geq 1$ and some $\gamma\in \Gamma$, then $\gamma g^i\gamma^{-1}\in C(s)=\langle s\rangle$, a contradiction as $g^i$ has infinite order. Therefore, we may find some $g_0\in \Gamma$ with $\langle g_0\rangle =\cup_{i\geq 1}C(g^i)$.  Clearly, $g\in \langle g_0\rangle$.

Next, since $\Gamma$ has trivial amenable radical, it possesses $P_\text{nai}$. As a result, we may find some $\gamma \in \Gamma$ such that $g$ is free from $\gamma g\gamma^{-1}$. Arguing similarly as in the proof of \thref{abstractcond}, we can conclude that $\mathbb{E}_{\cM}(\lambda(g))\in \bC\cdot \lambda(g)$.

Define $\Lambda=\langle g\in \Gamma:~\lambda(g^{-1})\mathbb{E}_{\cM}(\lambda(g))\neq 0,~\mbox{$g$ has infinite order}\rangle$. Clearly, $\Lambda\lhd \Gamma$ and $L(\Lambda)\subseteq \cM$. 

Suppose that $\Lambda\neq \{e\}$. In this case, we can appeal to the well-known fact that every nontrivial normal subgroup inside any hyperbolic groups with trivial amenable radical is automatically relative i.c.c. inside the ambient group, i.e., $\sharp\{h\gamma h^{-1}:~h\in \Lambda\}=\infty$ for any $e\neq \gamma\in \Gamma$. We refer the reader to the proof of \cite[Theorem 1.2]{A19} and \cite[Remark 3.8]{js2021maximal}. In particular, this tells us that $L(\Lambda)'\cap L(\Gamma)=\bC$. We can now apply \cite[Corollary 3.8]{chifan2020rigidity} to finish the proof.

Therefore, without any loss of generality, we may assume that $\Lambda$ is trivial. Equivalently, $\mathbb{E}_{\cM}(\lambda(g))=0$ for any $g\in \Gamma$ with infinite order. We are now left to show $\cM=\bC$. Since $Tor(\Gamma)=\{\gamma s\gamma^{-1}: \gamma\in \Gamma\}$, we only need to show that $\mathbb{E}_{\cM}(\lambda(s))=0$.

Now, we may write
$\mathbb{E}_{\cM}(\lambda(s))=\sum_{g\in \Gamma}a_{g sg^{-1}}\lambda(g sg^{-1})$ for some $a_{gsg^{-1}}\in \bC$. For any $\gamma\in \Gamma$, we have
\[\mathbb{E}_{\cM}(\lambda(\gamma s\gamma^{-1}))=\lambda(\gamma)\mathbb{E}_{\cM}(\lambda( s))\lambda(\gamma^{-1})=\sum_{g\in \Gamma}a_{g sg^{-1}}\lambda(\gamma g sg^{-1}\gamma^{-1}).\]
\noindent
Claim 2: $\mathbb{E}_{\cM}(\lambda(s)\cdot \mathbb{E}_{\cM}(\lambda(\gamma s\gamma^{-1})))=2a_{\gamma^{-1} s\gamma}$.

\noindent
\textit{Proof of Claim 2:}
Note that
\begin{align*}
    \mathbb{E}_{\cM}(\lambda(s)\mathbb{E}_{\cM}(\lambda(\gamma s\gamma^{-1})))
    =\sum_{g\in \Gamma}a_{gsg^{-1}}\mathbb{E}_{\cM}(\lambda(s\gamma gsg^{-1}\gamma^{-1})).
\end{align*}
We observe that $s\gamma gsg^{-1}\gamma^{-1}$ either has infinite order or is trivial. Suppose that it is not trivial. Then, it contains evenly many letter $s$ and hence, does not belong to $Tor(s)$. In particular, it has an infinite order. Moreover, observe that if $s\gamma gsg^{-1}\gamma^{-1}=e$, then $\gamma g \in C(s)=\langle s\rangle$.
In other words, $s\gamma gsg^{-1}\gamma^{-1}$ has infinite order iff $\gamma g\not\in \langle s\rangle$. It follows from our assumption that $\mathbb{E}_{\cM}(\lambda(s\gamma gsg^{-1}\gamma^{-1})=0$ whenever $\gamma g\not\in \langle s\rangle$, or equivalently, $g\not\in\{\gamma^{-1}, \gamma^{-1}s\}$. Clearly, this implies that $\mathbb{E}_{\cM}(\lambda(s)\mathbb{E}_{\cM}(\lambda(\gamma s\gamma^{-1})))=2a_{\gamma^{-1} s\gamma}$.

From Claim 2, we deduce that  $\mathbb{E}_{\cM}(\lambda(s))\lambda(\gamma)\mathbb{E}_{\cM}(\lambda(s))\lambda(\gamma^{-1})=2a_{\gamma^{-1} s\gamma}$ for all $\gamma\in \Gamma$.

In particular, by taking $\gamma=e$ and using the fact that $\mathbb{E}_{\cM}(\lambda(s))$ is self-adjoint, we deduce that $\mathbb{E}_{\cM}(\lambda(s))^2=2a_s\geq 0$.

Assume that $a_s\neq 0$. Then, $U:=\frac{\mathbb{E}_{\cM}(\lambda(s))}{\sqrt{2a_s}}$ is a self-adjoint unitary operator. Hence, $U\sqrt{2a_s}\lambda(\gamma)U\sqrt{2a_s}\lambda(\gamma^{-1})=2a_{\gamma^{-1} s\gamma}$, i.e.,
\begin{align*}
  \lambda(\gamma)U\lambda(\gamma^{-1})=\frac{a_{\gamma^{-1} s\gamma}}{a_s}U^*,~\forall~\gamma\in \Gamma.
\end{align*}
If we replace $\gamma$ by $s\gamma$, then the right hand side of the above identity does not change. Hence, we deduce that $\lambda(\gamma)U\lambda(\gamma^{-1})=\lambda(s\gamma)U\lambda(\gamma^{-1}s^{-1})$, i.e., $U$ commutes with $\lambda(\gamma^{-1}s\gamma)$ for all $\gamma\in \Gamma$. 

We now observe that $\{\gamma^{-1} s\gamma:~\gamma\in\Gamma\}$ is a relative i.c.c. set inside $\Gamma$, i.e.,
\[\sharp\{\gamma^{-1} s\gamma g\gamma^{-1} s\gamma\}=\infty,~\mbox{for all nontrivial $g\in \Gamma$.}\]  Indeed, by writing $\Gamma=F_2\rtimes_{\sigma} \frac{\bZ}{2\bZ}$, where $\sigma_s$ flips the two generators of $F_2$, it suffices to check that the following sets are infinite, i.e., for any $e\neq g'\in F_2$ and any $g''\in F_2$, \[\sharp\{\gamma^{-1}\sigma_s(\gamma)g'\sigma_s(\gamma)^{-1}\gamma:~\gamma\in F_2\}=\infty \text{ and  }\sharp\{\gamma^{-1}\sigma_s(\gamma)g''\gamma^{-1}\sigma_s(\gamma):~\gamma\in F_2\}=\infty.\] This is easy to verify by looking at the initial and ending letters of $g'$ and $g''$.

Therefore, we deduce that $U\in \bC$. Since $\tau(U)=\frac{1}{\sqrt{2a_s}}\tau(\lambda(s))=0$, we deduce that $U=0$, a contradiction to the fact that $U$ is an unitary element.

Hence, $a_s=0$. Therefore, $\mathbb{E}_{\cM}(\lambda(s))^2=2a_s=0$. Since $\mathbb{E}_{\cM}(\lambda(s))$ is a self-adjoint operator, we deduce that $\mathbb{E}_{\cM}(\lambda(s))=0$.
\end{proof}
\subsection*{Recent Developments} Question \ref{question: general case with trivial amenable radical} has recently been answered affirmatively in \cite{chifan2022invariant} using a different method. They have also showed that every i.c.c. acylindrically hyperbolic group and all i.c.c., nonamenable groups that have positive first $L^2$-Betti number and contain an infinite amenable subgroup, satisfy the ISR-property.

\section*{Acknowledgements}
The first named author's research is supported by the ISRAEL SCIENCE FOUNDATION
(Grant No. 1175/18). The second named author is supported by ``the Fundamental Research Funds for the Central Universities" (Grant No.DUT19RC(3)075).
The authors express their gratitude towards Adam Skalski for taking the time to read a near complete draft of this paper and for his numerous comments and suggestions which improved the exposition of this paper greatly. The authors also thank Adam Dor-On and Jesse Peterson for their helpful suggestions. 
The first named author is grateful near Mehrdad Kalantar and Yair Hartman for many helpful discussions regarding the problem. The first named author would also like to thank the officials (especially Kavita Ma'am) in the Israel Embassy, New Delhi. He was able to travel and start his postdoc position due to their timely help and intervention. We also thank the anonymous referee for his/her insightful comments and suggestions and for pointing out \thref{remark_Dixmier's work},  \thref{malnormal}, and \thref{necessary_condition_for_ISR}.

\end{document}